\documentclass[11pt]{amsart}
\usepackage{amssymb}
 \theoremstyle{plain}
\newtheorem{theorem}{Theorem}
\newtheorem{corollary}{Corollary}

\newtheorem{lemma}{Lemma}
\newtheorem{proposition}{Proposition}
\theoremstyle{example}
\newtheorem{example}{Example}
\theoremstyle{definition}
\newtheorem{definition}{Definition}
\theoremstyle{remark}

\numberwithin{equation}{section}

\newcommand{\bT}{\begin{theorem}}
\newcommand{\eT}{\end{theorem}}
\newcommand{\bProp}{\begin{proposition}}
\newcommand{\eProp}{\end{proposition}}
\newcommand{\bE}{\begin{example}}
\newcommand{\eE}{\end{example}}
\newcommand{\bL}{\begin{lemma}}
\newcommand{\eL}{\end{lemma}}
\newcommand{\bP}{\begin{proof}}
\newcommand{\eP}{\end{proof}}
\newcommand{\bC}{\begin{corollary}}
\newcommand{\eC}{\end{corollary}}
\newcommand{\bD}{\begin{definition}}
\newcommand{\eD}{\end{definition}}
\newcommand{\be}{\begin{enumerate}}
\newcommand{\ee}{\end{enumerate}}
\newcommand{\beqa}{\begin{eqnarray*}}
\newcommand{\eeqa}{\end{eqnarray*}}
\newcommand{\beqaa}{\begin{eqnarray}}
\newcommand{\eeqaa}{\end{eqnarray}}
\newcommand{\ba}{\begin{array}}
\newcommand{\ea}{\end{array}}

\setbox0=\hbox{$+$}
\newdimen\plusheight
\plusheight=\ht0
\def\+{\;\lower\plusheight\hbox{$+$}\;}

\setbox0=\hbox{$-$}
\newdimen\minusheight
\minusheight=\ht0
\def\-{\;\lower\minusheight\hbox{$-$}\;}

\setbox0=\hbox{$\cdots$}
\newdimen\cdotsheight
\cdotsheight=\plusheight
\def\cds{\lower\cdotsheight\hbox{$\cdots$}}
\begin{document}

\title[Some identities deriving from a new Bailey-type transformation]
       { Some identities between basic hypergeometric series deriving from a new Bailey-type transformation}
\author{James Mc Laughlin}
\address{Mathematics Department\\
 Anderson Hall\\
West Chester University, West Chester, PA 19383}
\email{jmclaughl@wcupa.edu}

\author{Peter Zimmer}
\address{Mathematics Department\\
 Anderson Hall\\
West Chester University, West Chester, PA 19383}
\email{pzimmer@wcupa.edu}


\date{\today}

\begin{abstract}
We prove a new Bailey-type transformation relating WP-Bailey pairs.
We then use this transformation to derive a number of new 3- and
4-term transformation formulae between basic hypergeometric series.
\end{abstract}

\maketitle

\section{Introduction}

Bailey's  transform can be stated as follows:
\begin{lemma}\label{l1}
Subject to suitable convergence conditions, if
\begin{equation*}
\beta_n = \sum_{r=0}^{n} \alpha_{r}U_{n-r}V_{n+r} \hspace{10pt}
\text{ and } \hspace{10pt} \gamma_n = \sum_{r=n}^{\infty}
\delta_{r}U_{r-n}V_{r+n},
\end{equation*}
then
\begin{equation*}
\sum_{n=0}^{\infty} \alpha_n \gamma_n= \sum_{n=0}^{\infty} \beta_n
\delta_n.
\end{equation*}
\end{lemma}
The proof follows by switching the order of summation (see
\cite{AAR99}, pages 583--584, for example). Bailey set
{\allowdisplaybreaks\begin{align*} U_n=\frac{1}{(q;q)_n},
\hspace{25pt} V_n=\frac{1}{(x;q)_n}, \hspace{25pt} \delta_n =
(y,z;q)_n \left ( \frac{x}{yz}\right )^n,
\end{align*}
}and  used the $q$-Gauss sum,
\begin{equation}\label{qgauss} _2\phi_1 (a,b;c;q,c/ab)=
\frac{(c/a,c/b;q)_{\infty}}{(c,c/ab;q)_{\infty}},
\end{equation}to get that
\begin{equation}\label{Baileyeq}
\sum_{n=0}^{\infty} (y,z;q)_{n}\left ( \frac{x}{yz}\right )^{n}
\beta_n = \frac{(x/y,x /z;q)_{\infty}}{ (x , x/yz;q)_{\infty}}
\sum_{n=0}^{\infty} \frac{(y,z;q)_{n}}{(x/y,x/z;q)_n}\left (
\frac{x}{yz}\right )^{n} \alpha_n,
\end{equation}
where $\alpha_0=1$ and
\begin{equation}\label{BP1}
\beta_n = \sum_{r=0}^{n} \frac{\alpha_r}{(q;q)_{n-r}(x;q)_{n+r}}.
\end{equation}

Here we are employing the usual notations. Let $a$ and $q$ be
complex numbers, with $|q|<1$ unless otherwise stated. Then
\begin{align*}
&(a)_0 =(a;q)_0 :=1, \hspace{20pt} (a)_n=(a;q)_n
:=\prod_{j=0}^{n-1}(1-a q^j), \text{ for } n\in \mathbb{N},\\
&(a_1;q)_n(a_2;q)_n \dots (a_k;q)_n = (a_1,a_2,\dots, a_k;q)_n,\\
&(a;q)_{\infty}:=\prod_{j=0}^{\infty}(1-a q^j), \hspace{20pt}\\
&(a_1;q)_{\infty}(a_2;q)_{\infty} \dots (a_k;q)_{\infty} =
(a_1,a_2,\dots, a_k;q)_{\infty}.
\end{align*}

An $_{r} \phi _{s}$ basic hypergeometric series is defined by
{\allowdisplaybreaks
\begin{multline*} _{r} \phi _{s} \left [
\begin{matrix}
a_{1}, a_{2}, \dots, a_{r}\\
b_{1}, \dots, b_{s}
\end{matrix}
; q,x \right ] = \\
\sum_{n=0}^{\infty} \frac{(a_{1};q)_{n}(a_{2};q)_{n}\dots
(a_{r};q)_{n}} {(q;q)_{n}(b_{1};q)_{n}\dots (b_{s};q)_{n}} \left(
(-1)^{n} q^{n(n-1)/2} \right )^{s+1-r}x^{n}.
\end{multline*}
}

In modern notation, the pair of sequences $(\alpha_n, \beta_n)$
above are termed a \emph{Bailey pair relative to $x/q$}. Slater, in
\cite{S51} and \cite{S52}, subsequently used this transformation of
Bailey to derive 130 identities of the Rogers-Ramanujan type.

The first major variations in Bailey's construct at \eqref{Baileyeq}
appear to be due to Bressoud \cite{B81a}. Another variation was
given by Singh in \cite{S94}. All of these variations were put in a
more formal setting by Andrews in \cite{A01}, where he introduced a
generalization of the standard Bailey pair as defined at \eqref{BP1}
(see also equation (9.3) in \cite{B47}).

\textbf{Definition.} (Andrews \cite{A01}) Two sequences
$(\alpha_{n}(a,k), \beta_{n}(a,k))$ form a \emph{WP-Bailey pair}
provided
%{\allowdisplaybreaks
\begin{equation}\label{WPpair}
\beta_{n}(a,k) = \sum_{j=0}^{n}
\frac{(k/a)_{n-j}(k)_{n+j}}{(q)_{n-j}(aq)_{n+j}}\alpha_{j}(a,k).
%\\
%&= \frac{(k/a,k;q)_n}{(aq,q;q)_n}\sum_{j=0}^{n}
%\frac{(q^{-n})_{j}(kq^n)_{j}}{(aq^{1-n}/k)_{j}(aq^{n+1})_{j}}
%\left(\frac{qa}{k}\right)^j\alpha_{j}(a,k). \notag
\end{equation}
%}
Note that if $k=0$, then the definition reverts to that of a
standard Bailey pair.

In the same paper Andrews showed  that there were two distinct ways
to construct new WP-Bailey pairs from a given pair. If
$(\alpha_{n}(a,k),\,\beta_{n}(a,k))$ satisfy \eqref{WPpair}, then so
do $(\alpha_{n}'(a,k),\,\beta_{n}'(a,k))$ and
$(\tilde{\alpha}_{n}(a,k),\,\tilde{\beta}_{n}(a,k))$, where
{\allowdisplaybreaks
\begin{align}\label{wpn1}
\alpha_{n}'(a,k)&=\frac{(\rho_1, \rho_2)_n}{(aq/\rho_1,
aq/\rho_2)_n}\left(\frac{k}{c}\right)^n\alpha_{n}(a,c),\\
\beta_{n}'(a,k)&=\frac{(k\rho_1/a,k\rho_2/a)_n}{(aq/\rho_1,
aq/\rho_2)_n} \notag\\
&\phantom{as}\times \sum_{j=0}^{n} \frac{(1-c
q^{2j})(\rho_1,\rho_2)_j(k/c)_{n-j}(k)_{n+j}}{(1-c)(k\rho_1/a,k\rho_2/a)_n(q)_{n-j}(qc)_{n+j}}
\left(\frac{k}{c}\right)^j\beta_{j}(a,c), \notag
\end{align}
}with  $c=k\rho_1 \rho_2/aq$ for the pair above, and
\begin{align}\label{wpn2}
\tilde{\alpha}_{n}(a,k)&= \frac{(qa^2/k)_{2n}}{(k)_{2n}}\left
(\frac{k^2}{q a^2} \right)^n\alpha_{n} \left(a, \frac{q a^2}{k}
\right), \\
\tilde{\beta}_{n}(a,k)&=\sum_{j=0}^{n}
\frac{(k^2/qa^2)_{n-j}}{(q)_{n-j}}\left (\frac{k^2}{q a^2}
\right)^j\beta_{j} \left(a, \frac{q a^2}{k} \right). \notag
\end{align}

Andrews two constructions can be shown to imply the following
Bailey-type transformations for WP-Bailey pairs, assuming suitable
convergence conditions.
\begin{theorem}
If $(\alpha_n(a,k),\beta_n(a,k))$ satisfy
\begin{equation*} \beta_{n}(a,k) = \sum_{j=0}^{n}
\frac{(k/a)_{n-j}(k)_{n+j}}{(q)_{n-j}(aq)_{n+j}}\alpha_{j}(a,k),
\end{equation*}
then
\begin{multline}\label{wpbteq1b}
\sum_{n=0}^{\infty} \frac{(1-kq^{2n})(\rho_1,\rho_2;q)_n }{(1-k)(k
q/\rho_1,k q/\rho_2;q)_n}\,\left(\frac{a q}{\rho_1 \rho_2}\right)^n
\beta_n(a,k)=\\\frac{(k q,k q/\rho_1\rho_2,a q/\rho_1,a
q/\rho_2;q)_{\infty}}{(k q/\rho_1,k q/\rho_2,a q/\rho_1 \rho_2,a
q;q)_{\infty}}
%\phantom{asdadasdasdabvvbvmvmbnvbnvdasdasdas}\\\times
\sum_{n=0}^{\infty} \frac{(\rho_1,\rho_2;q)_{n}} {(a q/\rho_1,a
q/\rho_2;q)_{n}}\left ( \frac{a q}{\rho_1 \rho_2}\right)^n
\alpha_n(a,k),
\end{multline}
and
\begin{multline}\label{wpbteq2b}
\sum_{n=0}^{\infty}  \left(\frac{q a^2}{k^2}\right)^n \beta_n(a,k)\\
= \frac{(qa/k,qa^2/k;q)_{\infty}} {(q a,qa^2/k^2;q)_{\infty}}
\sum_{n=0}^{\infty} \frac{(k;q)_{2n}} {(q a^2/k;q)_{2n}}\left (
\frac{q a^2}{k^2}\right)^n \alpha_n(a,k).
\end{multline}
\end{theorem}

We had initially derived the transformation at \eqref{wpbteq1b} in a
way that was similar to the way Bailey derived \eqref{Baileyeq},
before finding that it
 followed from Andrews' first construction at \eqref{wpn1}.
A result equivalent to the transformation at \eqref{wpbteq2b} was
also stated by Bressoud in \cite{B81a}.

In the present paper we prove the following transformation for
WP-Bailey pairs.
\begin{theorem}\label{t3}
Subject to suitable convergence conditions, if
\begin{equation}\label{betaneq3}
\beta_n = \sum_{r=0}^{n}
\frac{(k/a;q)_{n-r}}{(q;q)_{n-r}}\frac{(k;q)_{n+r}}{( a q;q)_{n+r}}
\alpha_{r},
\end{equation}
then {\allowdisplaybreaks
\begin{multline}\label{simsuma3}
\frac{ (q a b/k, k q/b;q)_{\infty}(q,k^2 q/a, q^2 a, q^2
a^2/k^2;q^2)_{\infty}}
 { (k q,q a/k;q)_{\infty}}\\
 \times  \sum_{n=0}^{\infty}
\frac{(q\sqrt{k}, -q\sqrt{k},k^2/ab,b,\sqrt{q a},- \sqrt{q
a};q)_n}{(\sqrt{k},-\sqrt{k}, q a b/k,k q/b, k \sqrt{q/a}, -
k\sqrt{q/a};q)_n}\left ( \frac{-q
a}{k}\right)^n \beta_n \\
= \left(\frac{q k^2}{ab}, b q,\frac{ q^2 a^2 b}{k^2}, \frac{q^2
a}{b};q^2\right)_{\infty}
%\\\times
\sum_{n=0}^{\infty}\frac{\left(\frac{k^2}{a b},b;q^2
\right)_{n}}{\left(\frac{q^2 a^2 b}{k^2}, \frac{q^2 a}{b};
q^2\right)_{n}} \left(\frac{-q a}{k}\right)^{2n} \alpha_{2n}
\phantom{asdsfasdf}\\
+ \left(\frac{k^2}{ab}, b,\frac{ q^3 a^2 b}{k^2}, \frac{q^3
a}{b};q^2\right)_{\infty}
%\\\times
\sum_{n=0}^{\infty}\frac{\left(\frac{k^2 q}{a b},b q;q^2
\right)_{n}}{\left(\frac{q^3 a^2 b}{k^2}, \frac{q^3 a}{b};
q^2\right)_{n}} \left(\frac{-q a}{k}\right)^{2n+1} \alpha_{2n+1}.
\end{multline}
}
\end{theorem}

We use this transformation to derive some new 3- and 4-term
transformations between basic hypergeometric series.

\section{A Transformation deriving from a $q$-analog of Watson's  $_3F_2$ sum}

We recall the following $q$-analogue of Watson's $_3 F_2$ sum (see
\cite[(II.16), page 355]{GR04}),
\begin{multline}
\label{waq3f2} _{8} \phi _{7} \left [
\begin{matrix}
\lambda, \,q \sqrt{\lambda},\, -q \sqrt{\lambda}, \,a,\, b,\,
\lambda \sqrt{q/a b},\,-\lambda \sqrt{q/a b}, \,
a b/\lambda\\
\sqrt{\lambda},\,-\sqrt{\lambda},\, \lambda q/a,\,\lambda
q/b,\,\lambda^2 q/a b,\,\sqrt{q a b},\, -\sqrt{q a b}
\end{matrix}
; q,-\frac{q \lambda}{a b}  \right ]\\ = \frac{(\lambda q,\lambda
q/a b;q)_{\infty}}{(\lambda q/a,\lambda q/b;q)_{\infty}}\frac{(a q,b
q,q^2 \lambda^2/a^2 b, q^2 \lambda^2/a b^2;q^2)_{\infty}}{( q, a b
q, q^2\lambda^2/a b,q^2\lambda^2/a^2b^2;q^2)_{\infty}}.
\end{multline}

We now prove the main theorem.

\begin{proof}[Proof of Theorem \ref{t3}.]
In Lemma \ref{l1},  set {\allowdisplaybreaks
\begin{align*}
U_r&=\frac{(ab/\lambda;q)_r}{(q;q)_r},\\
V_r&=\frac{(\lambda;q)_r}{(\lambda^2q/ab;q)_r},\\
\delta_r&=\frac{(q\sqrt{\lambda}, -q\sqrt{\lambda},a,b,\lambda
\sqrt{q/ab},-\lambda
\sqrt{q/ab};q)_r}{(\sqrt{\lambda},-\sqrt{\lambda}, \lambda
q/a,\lambda q/b, \sqrt{q a b}, -\sqrt{q a b};q)_r}\left ( \frac{-q
\lambda}{a b}\right)^r.
\end{align*}
} Then {\allowdisplaybreaks
 \begin{align*}
\gamma_n &= \sum_{r=n}^\infty \delta_r U_{r-n} V_{r+n}\\
&= \sum_{m=0}^\infty \delta_{m+n} U_{m} V_{m+2n}\\
&= \sum_{m=0}^\infty \frac{(q\sqrt{\lambda},
-q\sqrt{\lambda},a,b,\lambda \sqrt{q/ab},-\lambda
\sqrt{q/ab};q)_{m+n}} {(\sqrt{\lambda},-\sqrt{\lambda}, \lambda
q/a,\lambda q/b, \sqrt{q a b}, -\sqrt{q a b};q)_{m+n}}
\frac{(ab/\lambda;q)_{m}}{(q;q)_{m}}\\
&\phantom{asasdasdasdasdasdasdsdsdasdsd}\times\frac{(\lambda;q)_{m+2n}}
{(\lambda^2q/ab;q)_{m+2n}}\left(\frac{-q
\lambda}{a b}\right)^{m+n}\\
&= \frac{(q\sqrt{\lambda}, -q\sqrt{\lambda},a,b,\lambda
\sqrt{q/ab},-\lambda \sqrt{q/ab};q)_{n}}
{(\sqrt{\lambda},-\sqrt{\lambda}, \lambda q/a,\lambda q/b, \sqrt{q a
b}, -\sqrt{q a b};q)_{n}} \frac{(\lambda;q)_{2n}}
{(\lambda^2q/ab;q)_{2n}} \left(\frac{-q \lambda}{a b}\right)^{n}\\
&\times \sum_{m=0}^\infty \frac{(q^{1+n}\sqrt{\lambda}, -q^{1+n}
\sqrt{\lambda},a q^n,b q^n,\lambda \sqrt{q/ab}\,q^n,-\lambda
\sqrt{q/ab}\,q^n;q)_{m}} {(\sqrt{\lambda}\,q^n,-\sqrt{\lambda}\,q^n,
\lambda q^{1+n}/a,\lambda q^{1+n}/b, \sqrt{q a b}\,q^n, -\sqrt{q a
b}\,q^n;q)_{m}}\\
&\phantom{asasdasdasdasdasdasdsdsdasdsd}\times\frac{(\lambda
q^{2n},ab/\lambda;q)_{m}} {(\lambda^2q^{1+2
n}/ab,q;q)_{m}}\left(\frac{-q
\lambda}{a b}\right)^{m}\\
&= \frac{(q\sqrt{\lambda}, -q\sqrt{\lambda},a,b,\lambda
\sqrt{q/ab},-\lambda \sqrt{q/ab};q)_{n}}
{(\sqrt{\lambda},-\sqrt{\lambda}, \lambda q/a,\lambda q/b, \sqrt{q a
b}, -\sqrt{q a b};q)_{n}} \frac{(\lambda;q)_{2n}}
{(\lambda^2q/ab;q)_{2n}} \left(\frac{-q \lambda}{a b}\right)^{n}\\
&\times \frac{(\lambda q^{1+2n},\lambda q/a b;q)_{\infty}}{(\lambda
q^{1+n}/a,\lambda q^{1+n}/b;q)_{\infty}}\frac{(a q^{1+n},b
q^{1+n},q^{2+n} \lambda^2/a^2 b, q^{2+n} \lambda^2/a
b^2;q^2)_{\infty}}{( q, a b q^{1+2n}, q^2\lambda^{2+2n}/a
b,q^2\lambda^2/a^2b^2;q^2)_{\infty}}\\
&\phantom{as}\\ &=\frac{ (\lambda q, \lambda q/a b;q)_{\infty}}
 { (\lambda q/a,
\lambda q/b,q)_{\infty}} \frac{(a q^{1+n}, b q^{1+n},
\lambda^{2}q^{2+n}/a^2 b, \lambda^{2}q^{2+n}/a
b^2;q^2)_{\infty}}{(q,q a b, \lambda^{2}q^{2}/a b,
\lambda^{2}q^{2}/a^2b^2;q^2)_{\infty}}\\
&\phantom{asasdasdasdasdsadsdasdasdasdaasdasdsdsdasdsd}\times(a,b;q)_{n}\left(\frac{-q
\lambda}{a b}\right)^{n}\\ &=\frac{ (\lambda q, \lambda q/a
b;q)_{\infty}}
 { (\lambda q/a,
\lambda q/b,q)_{\infty}}\left(\frac{-q
\lambda}{a b}\right)^{n}\times \\
&\begin{cases}\displaystyle{ \frac{(a q, b q, \lambda^{2}q^{2}/a^2
b, \lambda^{2}q^{2}/a b^2;q^2)_{\infty}}{(q,q a b,
\lambda^{2}q^{2}/a b,
\lambda^{2}q^{2}/a^2b^2;q^2)_{\infty}}}\frac{(a,b;q^2)_{\frac{n}{2}}}{(\lambda^2
q^2/a^2 b, \lambda ^2 q^2/a b^2; q^2)_{\frac{n}{2}}}, & n
\text{ even},\\
\displaystyle{ \frac{(a, b, \lambda^{2}q^{3}/a^2 b,
\lambda^{2}q^{3}/a b^2;q^2)_{\infty}}{(q,q a b, \lambda^{2}q^{2}/a
b,
\lambda^{2}q^{2}/a^2b^2;q^2)_{\infty}}}\frac{(aq,bq;q^2)_{\frac{n-1}{2}}}{(\lambda^2
q^3/a^2 b, \lambda ^2 q^3/a b^2; q^2)_{\frac{n-1}{2}}}, & n \text{
odd}.
\end{cases}
\end{align*} } The fifth equality comes from applying
\eqref{waq3f2} to the sum from the line before, after replacing
$\lambda$ with $\lambda q^{2n}$, $a$ with $a q^n$ and $b$ with $b
q^n$ in this identity.

We next make the substitutions $\lambda \to k$, $a \to k^2/b c$
followed by $c \to a$, and again suppose the sequences
$\{\alpha_n\}$ and $\{\beta_n\}$ are related by
{\allowdisplaybreaks\begin{equation*}
\beta_n = \sum_{r=0}^{n} \alpha_{r}U_{n-r}V_{n+r}\\
= \sum_{r=0}^{n}
\frac{(k/a;q)_{n-r}}{(q;q)_{n-r}}\frac{(k;q)_{n+r}}{( a q;q)_{n+r}}
\alpha_{r}. \notag
\end{equation*}
}The result now follows.
\end{proof}

We can use this theorem to re-derive some known identities between
basic hypergeometric series, and also to derive some new identities.
For example, inserting the ``trivial" WP-Bailey pair
\begin{align*}
\alpha_{n}(a,k)&=
\begin{cases} 1&n=0,\\
0, &n>0,
\end{cases}\\
\beta_n(a,k)&=\frac{(k/a,k;q)_n}{(q,aq;q)_n}
\end{align*}
in \eqref{simsuma3} leads to a version of \eqref{waq3f2} (perhaps
not surprisingly, since \eqref{waq3f2} was used to prove Theorem
\ref{t3}). We also note in passing that applying Andrews first
construction at \eqref{wpn1} to this trivial WP-Bailey pair leads to
a variant of Jackson's sum of a terminating $\,_8 \phi_7$, while
applying his second construction at \eqref{wpn2} leads to a variant
of the $q$-Pfaff-Saalsch\"{u}tz sum.

\vspace{5pt}

Before going further, we introduce some standard space-saving
notation:
\[
_{r+1} W_{r}(a_1;a_4,\dots a_{r+1};q,z)=_{r+1}\phi_{r} \left [
\begin{matrix}
a_{1},\,q\sqrt{a_1},\,-q\sqrt{a_1},\,a_4,\,\dots, \, a_{r+1}\\
\sqrt{a_1},\,-\sqrt{a_1},\, \frac{a_1 q}{a_4},\,\dots , \frac{a_1
q}{a_{r+1}}
\end{matrix}
; q, z \right ].
\]

Inserting the ``unit" WP-Bailey pair (see \cite{AB02} for example,
where this WP-Bailey pair, and others employed below, may be found),
\begin{align*}
\alpha_{n}(a,k)&=\frac{(q \sqrt{a}, -q
\sqrt{a},a,a/k;q)_n}{(\sqrt{a},-\sqrt{a},q,kq;q)_n}\left(\frac{k}{a}\right)^n,\\
\beta_n(a,k)&=\begin{cases} 1&n=0,\\
0, &n>1,
\end{cases}
\end{align*}
in \eqref{simsuma3} and replacing $q$ with $\sqrt{q}$ leads to the
following identity:
\begin{multline*}
\frac{ \left(\displaystyle{\frac{\sqrt{q} a b}{k},\frac{q a
b}{k},\frac{k\sqrt{q}}{b}, \frac{k q}{b},-\frac{a\sqrt{q}}{k},
-\frac{a q}{k},\sqrt{q},\frac{k^2\sqrt{q}}{a}},q
a;q\right)_{\infty}}
 { (k q,k\sqrt{q};q)_{\infty}}\\
=\left(\frac{k^2\sqrt{q}}{a b},b\sqrt{q},\frac{a^2 b q}{k^2},\frac{
q a}{b};q\right)_{\infty} \, _8 W_{7}\left(a; a \sqrt{q},
\frac{a}{k},\frac{a \sqrt{q}}{k},\frac{k^2}{a b},b;q,q \right)\\
-q^{1/2}\frac{(1-a q)(1-a/k)}{(1-\sqrt{q})(1-k \sqrt{q})}
\left(\frac{k^2}{a b},b,\frac{a^2 b q^{3/2}}{k^2},\frac{ q^{3/2}
a}{b};q\right)_{\infty} \\
\times \, _8 W_{7}\left(a q; a \sqrt{q}, \frac{a
\sqrt{q}}{k},\frac{a q}{k},\frac{k^2 \sqrt{q}}{a b},b \sqrt{q};q,q
\right).
\end{multline*}
This identity is a particular case of Bailey's nonterminating
extension of Jackson's $_8 \phi_7$ sum (see \cite[page 356,
II.25]{GR04}).

We now state some transformations which we believe are new. Upon
substituting Singh's WP-Bailey pair \cite{S94} (see also
\cite{AB02}, where one of Andrews' constructions was used to derive
Singh's pair from the unit pair),
\begin{align}\label{singhpr}
\alpha_{n}(a,k)&=\frac{(q \sqrt{a}, -q
\sqrt{a},a,y,z,a^2q/kyz;q)_n}
{(\sqrt{a},-\sqrt{a},q,a q/y,a q/z,kyz/a;q)_n}\left(\frac{k}{a}\right)^n,\\
\beta_n(a,k)&=\frac{(k y/a, kz/a, k, aq/yz;q)_n}{(a q/y, a q/z, k y
z/a,q;q)_n}, \notag
\end{align}
into \eqref{simsuma3}, we get the following transformation.

\begin{corollary}\label{singhcor}
\begin{multline}\label{singhcoreq}
\frac{ (q a b/k, k q/b;q)_{\infty}(q,k^2 q/a, q^2 a, q^2
a^2/k^2;q^2)_{\infty}}
 { (k q,q a/k;q)_{\infty}}\\
 \times\, _{10}W_{9}
 \left(k;\displaystyle{\frac{k^2}{a b},b,\sqrt{q a},-\sqrt{q a},
 \frac{k y}{a}, \frac{k z}{a}, \frac{a q}{y z};\,q, \,-\frac{q a}{k}} \right)\\
=\left(\frac{k^2q}{a b},b q,\frac{a^2 b q^2}{k^2},\frac{ q^2
a}{b};q^2\right)_{\infty}\phantom{dsasdasdasdasdasdsaasdfdasdadsadaa}\\
\times \, _{12} W_{11} \left(a; \frac{k^2}{a b},b,a q,y,y q,z, z q,
\frac{a^2 q}{k y z},
\frac{a^2 q^2}{k y z} ;\,q^2,\,q^2 \right)\\
-q \frac{(1-a q^2)(1-y)(1-z)(1-a^2 q/k y z)} {(1-q)(1-a q/y)(1-a
q/z)(1-k y z/a)} \left(\frac{k^2}{a b},b,\frac{a^2 b
q^3}{k^2},\frac{ q^3
a}{b};q^2\right)_{\infty}\\
\times \, _{12} W_{11} \left(q^2 a; \frac{k^2 q}{a b},b q,a q,y q,y
q^2,z q, z q^2, \frac{a^2 q^2}{k y z}, \frac{a^2 q^3}{k y z}
;\,q^2,\,q^2 \right).
\end{multline}
\end{corollary}
Remark: The identity above may be regarded as an extension
 of \eqref{waq3f2}, since substituting $y=1$ in  \eqref{singhcoreq}
 leads to a variant of \eqref{waq3f2}.

We next apply the theorem to some WP-Bailey pairs found by Andrews
and Berkovich \cite{AB02}.

\begin{corollary}\label{ab1}
{\allowdisplaybreaks
\begin{multline}\label{ab1eqeq}
\frac{ (q a b/k, k q/b;q)_{\infty}(q,k^2 q/a, q^2 a, q^2
a^2/k^2;q^2)_{\infty}}
 { (k q,q a/k;q)_{\infty}}\\
 \\
 \times \,_{7} \phi _{6} \left [
\begin{matrix}
q\sqrt{k},\,-q\sqrt{k}, \frac{k^2}{a b},\,b,\,\sqrt{q a},
\,-\sqrt{q a},\, \frac{k^2}{q a^2}\\
\sqrt{k},\,-\sqrt{k},\,\frac{q a b}{k},\,\frac{q k}{b},\,
k\sqrt{\frac{q}{a}},\,-k\sqrt{\frac{q}{a}}
\end{matrix}
; q,-\frac{q a}{k} \right ] \\\\
=\left(\frac{k^2q}{a b},b q,\frac{a^2 b q^2}{k^2},\frac{ q^2
a}{b};q^2\right)_{\infty}\phantom{dsasdasdasdasdasdsaasdfdasdadsadaa}\\
\text{\scriptsize{$\times \, _{16} W_{15} \left(a; \frac{k^2}{a
b},b,a q,\frac{k}{a q},\frac{k}{a},a \sqrt{\frac{q}{k}},-a
\sqrt{\frac{q}{k}}, \frac{a q}{\sqrt{k}}, -\frac{a q}{\sqrt{k}} ,a
\sqrt{\frac{q^3}{k}},-a \sqrt{\frac{q^3}{k}}, \frac{a
q^2}{\sqrt{k}},
-\frac{a q^2}{\sqrt{k}};\,q^2,\,q^2 \right)$}}\\
-q \frac{(1-a q^2)(1-q a^2/k)(1-q^2 a^2/k)(1-k/a q)} {(1-q)(1-k)(1-k
q)(1-a^2 q^2/k)} \left(\frac{k^2}{a b},b,\frac{a^2 b
q^3}{k^2},\frac{ q^3
a}{b};q^2\right)_{\infty}\times\\
\text{\tiny{$\,_{16}W_{15}\left(aq^2;\frac{k^2q}{ab},bq,aq,\frac{kq}{a},\frac{k}{a}
,a\sqrt{\frac{q^3}{k}},-a\sqrt{\frac{q^3}{k}},\frac{aq^2}{\sqrt{k}},\frac{-aq^2}{\sqrt{k}},
a\sqrt{\frac{q^5}{k}},-a\sqrt{\frac{q^5}{k}},\frac{aq^3}{\sqrt{k}},\frac{-aq^3}{\sqrt{k}};q^2,q^2\right)$}}.
\end{multline}
}
\end{corollary}

\begin{proof}
Insert the WP-Bailey pair
\begin{align*}
\alpha_n(a,k)&=\frac{(a,q\sqrt{a},-q \sqrt{a}, k/aq;q)_n}
{(q,\sqrt{a}, -\sqrt{a}, a^2 q^2/k;q)_n}\frac{(q
a^2/k;q)_{2n}}{(k;q)_{2n}}
\left(\frac{k}{a} \right)^n,\\
\beta_n(a,k)&=\frac{(k^2/q a^2;q)_n}{(q;q)_n},
\end{align*}
from \cite{AB02} into \eqref{simsuma3}.
\end{proof}

\begin{corollary}\label{ab2}
{\allowdisplaybreaks
\begin{multline}\label{ab2eqeq}
\frac{ (q a b/k, k q/b;q)_{\infty}(q,k^2 q/a, q^2 a, q^2
a^2/k^2;q^2)_{\infty}}
 { (k q,q a/k;q)_{\infty}}\\
 \times \,_{6} \phi _{5} \left [
\begin{matrix}
-q\sqrt{k}, \frac{k^2}{a b},\,b,\,\sqrt{q a},
\,-\sqrt{q a},\, \frac{k^2}{ a^2}\\
-\sqrt{k},\,\frac{q a b}{k},\,\frac{q k}{b},\,
k\sqrt{\frac{q}{a}},\,-k\sqrt{\frac{q}{a}}
\end{matrix}
; q,-\frac{q a}{k} \right ] \\
=\left(\frac{k^2q}{a b},b q,\frac{a^2 b q^2}{k^2},\frac{ q^2
a}{b};q^2\right)_{\infty}\times \phantom{dsasdasdasdasdasdsaasdfdasdadsa}\\
\text{\scriptsize{$\, _{16} W_{15} \left(a; \frac{k^2}{a b},b,a
q,\frac{k q}{a},\frac{k}{a},a \sqrt{\frac{q}{k}},-a
\sqrt{\frac{q}{k}},a \sqrt{\frac{q^3}{k}},-a \sqrt{\frac{q^3}{k}},
 \frac{a }{\sqrt{k}}, \frac{a q}{\sqrt{k}} ,-\frac{a q}{\sqrt{k}},
-\frac{a q^2}{\sqrt{k}};\,q^2,\,q^2 \right)$}}\\
-q\frac{(1-aq^2)\left(1-\frac{k}{a}\right)\left(1-\frac{a}{\sqrt{k}}\right)\left(1+\frac{a
q}{\sqrt{k}}\right)}{(1-q)(1-k q)(1-\sqrt{k}
q)(1+\sqrt{k})}
%\times\phantom{dsasdasdasdas}
%\\
\left(\frac{k^2}{a b},b,\frac{a^2 b q^3}{k^2},\frac{ q^3
a}{b};q^2\right)_{\infty}\times\\
\text{\tiny{$_{16}W_{15}\left(aq^2;\frac{k^2q}{ab},bq,aq,\frac{kq}{a},\frac{kq^2}{a}
,a\sqrt{\frac{q^3}{k}},-a\sqrt{\frac{q^3}{k}},a\sqrt{\frac{q^5}{k}},-a\sqrt{\frac{q^5}{k}},
\frac{aq}{\sqrt{k}},\frac{aq^2}{\sqrt{k}},
\frac{-aq^2}{\sqrt{k}},\frac{-aq^3}{\sqrt{k}};q^2,q^2\right)$}}.
\end{multline}
}
\end{corollary}

\begin{proof}
This transformation follows similarly, after inserting the WP-Bailey
pair
\begin{align*}
\alpha_n(a,k)&=\frac{\left(a,q\sqrt{a},-q \sqrt{a},
a\sqrt{\frac{q}{k}},-a\sqrt{\frac{q}{k}},\frac{a}{\sqrt{k}},-\frac{a
q}{\sqrt{k}},\frac{k}{a};q\right)_n} {\left(q,\sqrt{a}, -\sqrt{a},
\sqrt{q k},-\sqrt{q k}, q \sqrt{k}, -\sqrt{k}, \frac{q a^2}{k}
;q\right)_n}
\left(\frac{k}{a} \right)^n,\\
\beta_n(a,k)&=\frac{\left(\sqrt{k},\frac{k^2}{a^2};q\right)_n}{(q, q
\sqrt{k};q)_n},
\end{align*}
from \cite{AB02} into \eqref{simsuma3}.
\end{proof}

We next consider two WP-Bailey pairs found by Bressoud \cite{B81a}
(see also \cite{AB02}, where these pairs are also investigated). We
do not consider Bressoud's first WP-Bailey pair, since, as remarked
in \cite{AB02}, it is a limiting case of Singh's WP-Bailey pair at
\eqref{singhpr}.

\begin{corollary}\label{br1}
\begin{multline}\label{br1eq}
\frac{ (q a b/k, k q/b;q)_{\infty}(q,k^2 q/a, q^2 a, q^2
a^2/k^2;q^2)_{\infty}}
 { (k q,q a/k;q)_{\infty}}\\
\phantom{dsasdasdasdasdasda} \times \,_{8} W _{7}
\left(k;\frac{k^2}{a b},b,\sqrt{a q}, \frac{a q}{k},
 \frac{-k}{\sqrt{a}};q, -\sqrt{q} \right)\\
=\left(\frac{k^2q}{a b},b q,\frac{a^2 b q^2}{k^2},\frac{ q^2
a}{b};q^2\right)_{\infty}\phantom{dsasdasdasdasdasdsaasdfdasdadsa}\\
\text{\normalsize{$\times \, _{10} W_{9} \left(\sqrt{a};
\frac{k}{\sqrt{a b}},\frac{-k}{\sqrt{a b}}, \sqrt{b},
-\sqrt{b},\sqrt{a q},\frac{a
\sqrt{q}}{k},\frac{a q}{k};\,q,\,q \right)$}}\\
-\sqrt{q}\frac{(1-q \sqrt{a})\left(1-\frac{a\sqrt{q} }{k}\right)}
{(1-\sqrt{q})(1-\frac{k}{\sqrt{a}})} \left(\frac{k^2}{a
b},b,\frac{a^2 b q^3}{k^2},\frac{ q^3
a}{b};q^2\right)_{\infty}\phantom{dsasdasdasdasdasda}\\
\text{\normalsize{$\times\,_{10} W_{9} \left(q\sqrt{a};
k\sqrt{\frac{q}{a b}},-k\sqrt{\frac{q}{a b}}, \sqrt{b q}, -\sqrt{b
q},\sqrt{a q},\frac{a \sqrt{q^3}}{k},\frac{a q}{k};\,q,\,q
\right)$}}.
\end{multline}
\end{corollary}

\begin{proof}
Insert Bressoud's second WP-Bailey pair
\begin{align*}
\alpha_n(a,k)&=\frac{1-\sqrt{a}\,q^n}{1-\sqrt{a}}\,
\frac{\left(\sqrt{a},\frac{a
\sqrt{q}}{k};\sqrt{q}\right)_n} {\left(\sqrt{q},\frac{k}{
\sqrt{a}};\sqrt{q}\right)_n}
\left(\frac{k}{a \sqrt{q}} \right)^n,\\
\beta_n(a,k)&=\frac{\left(k,\frac{a q}{k};q\right)_n}{\left(q,
\frac{k^2}{a};q\right)_n}
\frac{\left(\frac{-k}{\sqrt{a}};\sqrt{q}\right)_{2n}}{\left( -
\sqrt{a q};\sqrt{q}\right)_{2n}}\left(\frac{k}{a \sqrt{q}}
\right)^n,
\end{align*}
into \eqref{simsuma3}.
\end{proof}

\begin{corollary}\label{br2}
\begin{multline}\label{br2eq}
\frac{ (q a b/k, k q/b;q)_{\infty}(q,k^2 q/a, q^2 a, q^2
a^2/k^2;q^2)_{\infty}}
 { (k q,q a/k;q)_{\infty}}\\
\phantom{dsasdasdasdasdasda} \times \,_{8} W _{7}
\left(k;\frac{k^2}{a b},b,\sqrt{a q}, \frac{a}{k},
 \frac{-k q}{\sqrt{a}};q, -\sqrt{q} \right)\\
=\left(\frac{k^2q}{a b},b q,\frac{a^2 b q^2}{k^2},\frac{ q^2
a}{b};q^2\right)_{\infty}\times \phantom{dsasdasdasdasdasdsaasdfghdasdad}\\
\text{\normalsize{$ \, _{12} W_{11} \left(\sqrt{a} ; \,i \,q\,
a^{1/4},-i\, q\, a^{1/4},\frac{k}{\sqrt{a b}},\frac{-k}{\sqrt{a b}},
\sqrt{b}, -\sqrt{b},\sqrt{a q},\frac{a}{k},\frac{a
\sqrt{q}}{k};\,q,\,q \right)$}}\\
-\sqrt{q}\frac{(1-a q^2 )\left(1-\frac{a }{k}\right)}
{(1-\sqrt{q})\left(1-k\sqrt{\frac{q}{a}}\right)(1+\sqrt{a})}
\left(\frac{k^2}{a b},b,\frac{a^2 b q^3}{k^2},\frac{ q^3
a}{b};q^2\right)_{\infty}\times \phantom{dsasdasdasa}\\
\text{\footnotesize{$\,_{12} W_{11} \left(q\sqrt{a}; \,i \,q^{3/2}
a^{1/4},-i\, q^{3/2} a^{1/4},k\sqrt{\frac{q}{a
b}},-k\sqrt{\frac{q}{a b}}, \sqrt{b q}, -\sqrt{b q},\sqrt{a
q},\frac{a \sqrt{q}}{k},\frac{a q}{k};\,q,\,q \right)$}}.
\end{multline}
\end{corollary}

\begin{proof}
Insert Bressoud's third WP-Bailey pair
\begin{align*}
\alpha_n(a,k)&=\frac{1-a\,q^{2n}}{1-a}\,
\frac{\left(\sqrt{a},\frac{a}{k};\sqrt{q}\right)_n}
{\left(\sqrt{q},k\sqrt{\frac{q}{ a}};\sqrt{q}\right)_n}
\left(\frac{k}{a \sqrt{q}} \right)^n,\\
\beta_n(a,k)&=\frac{\left(k,\frac{a}{k},-k\sqrt{\frac{q}{a}},-\frac{k
q}{\sqrt{a}};q\right)_n}{\left(q, \frac{q k^2}{a},-\sqrt{a},-\sqrt{a
q};q\right)_n} \left(\frac{k}{a \sqrt{q}} \right)^n,
\end{align*}
into \eqref{simsuma3}.
\end{proof}

Finally, we apply the theorem to three WP-Bailey pairs found by the
present authors in \cite{MZ07b}:
\begin{align}\label{mz01}
\alpha_n^{(1)}(a,k)&=\frac{(q a^2/k^2;q)_n}{(q,q)_n}\left(
\frac{k}{a}\right)^n,\\
\beta_n^{(1)}(a,k)&=\frac{(q
a/k,k;q)_n}{(k^2/a,q,q)_n}\frac{(k^2/a;q)_{2n}}{(a q,q)_{2n}}.\notag
\end{align}

\begin{align}\label{mz02}
\alpha_n^{(2)}(a,k)&=\frac{(a, \,q \sqrt{a},\, -q \sqrt{a},
\,k/a,\,a \sqrt{q/k}, \, -a \sqrt{q/k};q)_n}
{(\sqrt{a},\,-\sqrt{a},\, q
a^2/k,\,\sqrt{qk},\, -\sqrt{qk},\,q;q)_n}\,(-1)^n,\\
\beta_n^{(2)}(a,k)&=\begin{cases}
\displaystyle{\frac{(k,k^2/a^2;q^2)_{n/2}}{(q^2, q^2
a^2/k;q^2)_{n/2}}}, & n \text{ even},\\
0,& n \text{ odd}.
\end{cases}\notag
\end{align}

\begin{align}\label{mz03}
\alpha_n^{(3)}(a,q)&=\frac{(a, \,q \sqrt{a},\, -q \sqrt{a}, \,d,\,
q/d, \, -a;q)_n} {(\sqrt{a},\,-\sqrt{a},\, aq/d,\,a d,\,
-q,\,q;q)_n}\,(-1)^n,\\
\beta_n^{(3)}(a,q)&=\begin{cases}
\displaystyle{\frac{(q^2/ad,dq/a;q^2)_{n/2}}{(adq,aq^2/d;q^2)_{n/2}}}, & n \text{ even},\\
\displaystyle{-a\frac{(q/ad,d/a;q^2)_{(n+1)/2}}{(ad,aq/d;q^2)_{(n+1)/2}}},&
n \text{ odd}.
\end{cases}\notag
\end{align}

Note that the third pair is restricted in the sense that it is
necessary to set $k=q$ for \eqref{betaneq3} to hold.

\begin{corollary}\label{mz1}
{\allowdisplaybreaks
\begin{multline}\label{mz1eq}
\frac{ (q a b/k, k q/b;q)_{\infty}(q,k^2 q/a, q^2 a, q^2
a^2/k^2;q^2)_{\infty}}
 { (k q,q a/k;q)_{\infty}}\\
\phantom{dsasdasdasdasdasda} \times \,_{8} W _{7}
\left(k;\frac{k^2}{a b},b, \frac{q a}{k},
 \frac{k}{\sqrt{a}},
 \frac{-k}{\sqrt{a}};q, \frac{-q a}{k} \right)\\
=\left(\frac{k^2q}{a b},b q,\frac{a^2 b q^2}{k^2},\frac{ q^2
a}{b};q^2\right)_{\infty} \, _{4} \phi_{3}
\left [
\begin{matrix}
\frac{k^2}{a b},\,b,\frac{q a^2}{ k^2},\frac{q^2 a^2}{ k^2}\\
\\
\frac{q^2 a^2 b}{k^2},\,\frac{q^2 a}{b},\, q
\end{matrix}
; q^2,q^2 \right ]\phantom{dsasdasdasdas} \\
-q\frac{\left(1-\frac{q a^2 }{k^2}\right)} {(1-q)}
\left(\frac{k^2}{a b},b,\frac{a^2 b q^3}{k^2},\frac{ q^3
a}{b};q^2\right)_{\infty} \, _{4} \phi_{3}\left [
\begin{matrix}
\frac{k^2q}{a b},\,b q,\frac{q^2 a^2}{ k^2},\frac{q^3 a^2}{ k^2}\\
\\
\frac{q^3 a^2 b}{k^2},\,\frac{q^3 a}{b},\, q^3
\end{matrix}
; q^2,q^2 \right ].
\end{multline}
}
\end{corollary}

\begin{proof}
Insert the WP-Bailey pair at \eqref{mz01} into \eqref{simsuma3}.
\end{proof}

Remark: The substitution $k=a\sqrt{q}$ also gives a special case of
\eqref{waq3f2}.

\begin{corollary}\label{mz2}
\begin{multline}\label{mz2eq}
\frac{ (q a b/k, k q/b;q)_{\infty}(q,k^2 q/a, q^2 a, q^2
a^2/k^2;q^2)_{\infty}}
 { (k q,q a/k;q)_{\infty}}\times \\
_{12} W _{11} \left(k;\frac{k^2}{a b},\frac{k^2 q}{a b},b,b
q,\sqrt{a q},
-\sqrt{a q}, \sqrt{a q^3}, -\sqrt{a q^3}, \frac{k^2}{a^2};q^2, \frac{a^2q^2}{k^2} \right)\\
=\left(\frac{k^2q}{a b},b q,\frac{a^2 b q^2}{k^2},\frac{ q^2
a}{b};q^2\right)_{\infty}\times \phantom{dsasdasdasdasdasdsaasdfdaadsa}\\
\text{\normalsize{$ _{12} W_{11} \left(a; \frac{k^2}{a b},b,a
q,\frac{k}{a},\frac{k q}{a},
a\sqrt{\frac{q}{k}},-a\sqrt{\frac{q}{k}},
a\sqrt{\frac{q^3}{k}},-a\sqrt{\frac{q^3}{k}};q^2,\frac{q^2a^2}{k^2} \right)$}}\\
+\frac{q a}{k}\frac{(1-a q^2)\left(1-\frac{k }{a}\right)} {(1-q)(1-k
q)} \left(\frac{k^2}{a b},b,\frac{a^2 b q^3}{k^2},\frac{ q^3
a}{b};q^2\right)_{\infty}\times\phantom{dsasdasdasddasda}\\
\text{\small{$_{12}
W_{11}\left(aq^2;\frac{k^2q}{ab},bq,aq,\frac{kq}{a},\frac{kq^2}{a},
a\sqrt{\frac{q^3}{k}},-a\sqrt{\frac{q^3}{k}},
a\sqrt{\frac{q^5}{k}},-a\sqrt{\frac{q^5}{k}};q^2,\frac{q^2a^2}{k^2}
\right)$}}.
\end{multline}
\end{corollary}

\begin{proof}
Insert the WP-Bailey pair at \eqref{mz02} into \eqref{simsuma3}.
\end{proof}

\begin{corollary}\label{mz3}
{\allowdisplaybreaks
\begin{multline}\label{mz3eq}
\frac { \left(a b,
\displaystyle{\frac{q^2}{b}};q\right)_{\infty}\left(q,
\displaystyle{\frac{q^3}{a}}, q^2 a, a^2;q^2\right)_{\infty}}
{\left
( q^2, a;q\right)_{\infty}}\times\\
\Bigg[\,_{14}W_{13} \left ( q;\frac{q^2}{ab}, \frac{q^3}{ab},b,bq,
\sqrt{qa},-\sqrt{ qa}, \sqrt{q^3a},-\sqrt{ q^3a},
\frac{q^2}{ad},\frac{qd}{a},q^2; q^2,a^2 \right )
\\
+ a^2\frac{(1-q^3)\left( 1-\frac{q^2}{ab}\right)(1-b)(1-aq)
\left(1-\frac{q}{ad} \right)\left(1-\frac{d}{a}\right)} {(1-q)\left(
1-\frac{q^2}{b}\right)(1-ab)(1-ad)
\left(1-\frac{q^3}{a} \right)\left(1-\frac{a q}{d}\right)}\times\\
\text{\small{$_{14}W_{13}\left(q^3;\frac{q^3}{ab},\frac{q^4}{ab},bq,bq^2,
\sqrt{q^3a},-\sqrt{ q^3a},\sqrt{q^5a},-\sqrt{q^5a},
\frac{q^3}{ad},\frac{q^2d}{a},q^2; q^2,a^2\right)$}}\Bigg]
\\
    =  \left(a^2 b, b q,
\displaystyle{\frac{q^{3}}{a b}, \frac{q^{2}a}{
b}};q^2\right)_{\infty}
%\times \,
%\phantom{asdadsdasdasdasdasd}\\
%\phantom{asdadsdasdasdasdasd}
 \,_{12}W_{11} \left (
a;\frac{q^2}{ab},b, aq, -a, -aq, d,dq, \frac{ q}{d},\frac{q^2}{d}
; q^2,a^2 \right )\\
+ a\frac{(1-aq^2)(1-d)\left(1-\frac{q}{d}\right)(1+a)}
{(1-q^2)\left(1-\frac{a q}{d}\right)(1-ad)}\left(qa^2 b, b,
\displaystyle{\frac{q^{2}}{a b}, \frac{q^{3}a}{
b}};q^2\right)_{\infty}
\times\, \\
_{12}W_{11} \left ( aq^2;\frac{q^3}{ab},bq, aq, -aq, -aq^2, dq,dq^2,
\frac{q^2}{d},\frac{q^3}{d} ; q^2,a^2 \right ).
\end{multline}
}
\end{corollary}
\begin{proof}
Insert the WP-Bailey pair at \eqref{mz03} into \eqref{simsuma3}, and
set $k=q$.
\end{proof}

Remark:  The extra $q$-products  inserted in the numerators and
denominators of the terms in the two series on the left in the
identity above are there so as to give each these series the form of
a $_{r+1}W_r$ series.

%\allowdisplaybreaks{

%}
\end{document}